\documentclass{amsart}

\usepackage{amsmath,amssymb,amsfonts}
\usepackage[all]{xy}
\usepackage{appendix,listings,hyperref}

\DeclareMathAlphabet{\mathpzc}{OT1}{pzc}{m}{it}

\DeclareMathOperator{\rank}{rk}

\DeclareMathOperator{\Hom}{Hom}
\DeclareMathOperator{\Sym}{Sym}

\DeclareMathOperator{\len}{len}
\DeclareMathOperator{\discr}{discr}

\newcommand{\DP}{\,{:}\,}
\newcommand{\ie}{{\it i.e. }}

\newcommand{\defeq}{\sim_{\text{def}}}

\newcommand{\p}[2]{p_{#1}^{#2}\;\!\!}
\renewcommand{\L}{\mathcal{L}}

\newcommand{\coloneqq}{:=}
\newcommand{\bra}{\left<\!\!\!\:\left<}
\newcommand{\ket}{\right>\!\!\!\:\right>}
\newcommand{\myeq}[1]{\mathrel{\overset{\makebox[0pt]{\text{\tiny #1}}}{=}}}

\newcommand{\G}{\mathbb{G}}
\newcommand{\R}{\mathbb{R}}
\newcommand{\Q}{\mathbb{Q}}
\newcommand{\Z}{\mathbb{Z}}
\renewcommand{\S}{\mathbb{S}}

\newcommand{\Har}{\mathpzc{H}}

\theoremstyle{plain}
\newtheorem{theorem}{Theorem}[section]
\newtheorem{lemma}[theorem]{Lemma}
\newtheorem{proposition}[theorem]{Proposition}
\newtheorem{corollary}[theorem]{Corollary}
\theoremstyle{definition}
\newtheorem{definition}[theorem]{Definition}

\theoremstyle{remark}
\newtheorem{remark}[theorem]{Remark}
\newtheorem{example}[theorem]{Example}

\allowdisplaybreaks

\begin{document}

\title[Symmetric Powers, Hom.~Orth.~Polynomials, Hyperk\"ahlers]{Symmetric Powers of Symmetric Bilinear Forms, Homogeneous Orthogonal Polynomials on the Sphere and an Application to Compact Hyperk\"ahler Manifolds}

\author{Simon Kapfer}
\address{Simon Kapfer, Laboratoire de Math\'ematiques et Applications, UMR CNRS 6086, Universit\'e de Poitiers, T\'el\'eport 2, Boulevard Marie et Pierre Curie, F-86962 Futuroscope Chasseneuil}
\email{simon.kapfer@math.univ-poitiers.fr}

\date{\today}
\subjclass[2010]{14C05, 15A63, 33C50}
\keywords{Symmetric Bilinear Forms on Symmetric Powers, Orthogonal Polynomials in Several Variables, Homogeneous Orthogonal Polynomials, Gegenbauer Polynomials, Ultraspherical Polynomials, Hermite Polynomials, Spherical Harmonics, Hankel matrices, Hyperk\"ahler Manifolds, Irreducible Holomorphic Symplectic Manifolds, Beauville--Bogomolov Form, Beauville--Fujiki relation}

\begin{abstract} The Beauville--Fujiki relation for a compact Hyperk\"ahler manifold $X$ of dimension $2k$ allows to equip the symmetric power $\Sym^kH^2(X)$ with a symmetric bilinear form induced by the Beauville--Bogomolov form. We study some of its properties and compare it to the form given by the Poincar\'e pairing.

The construction generalizes to a definition for an induced symmetric bilinear form on the symmetric power of any free module equipped with a symmetric bilinear form. We point out how the situation is related to the theory of orthogonal polynomials in several variables.
Finally, we construct a basis of homogeneous polynomials that are orthogonal when integrated over the unit sphere $\S^d$, or equivalently, over $\R^{d+1}$ with a Gaussian kernel.
\end{abstract}

\maketitle


\section{Introduction}
Our motivation originated in Hyperk\"ahler theory. The Beauville--Bogomolov--Fujiki form $q$ for a compact Hyperk\"ahler manifold $X$ is a quadratic form on the integral cohomology $H^2:=H^2(X,\Z)$, defined by an equation of the structure
\begin{equation} \label{initialeq}
q(x)^k = I(x^{2k}),
\end{equation}
where $x^{2k}$ means a power in the cohomology ring, and $I$ is a linear form (in fact, a scaled integral). 

Now every quadratic form $q$ has an associated symmetric bilinear form $\left<\ ,\;\right>$, obtained by polarization: $2\left<x,y\right> = q(x+y)-q(x)-q(y)$. This allows us to retrieve some information about $I$ from $\left<\ ,\;\right>$,
by comparing coefficients in the equality
\begin{equation*}
I\!\left((x_1+\ldots+x_{2k})^{2k}\right) = q(x_1+\ldots+x_{2k})^{k} = \left(\sum\limits_{i=1}^{2k} q(x_i) +\!\!\!\sum\limits_{1\leq i<j\leq 2k}\!\!2 \left<x_i,x_j\right>  \right)^k.
\end{equation*}
If we look at the summands belonging to the monomial $x_1\ldots x_{2k}$, we obtain a seemingly more general but in fact equivalent version of (\ref{initialeq}):
\begin{equation} \label{initialpolar}
 (2k)!\; I\!\left(x_1\ldots x_{2k}\right) = 2^k k!\sum_{\mathcal{P}} \prod_{\{i,j\}\in\mathcal{P}}\left<x_i,x_j\right>,
\end{equation}
where the sum is over all partitions $\mathcal{P}$ of $\{1,\ldots,2k\}$ into pairs. This is a classical observation, see also \cite[Eq.~3.2.4]{OGrady}. Let us develop this idea a bit further. The map $(f,g) \mapsto I(fg)$ clearly defines a symmetric bilinear form on the symmetric product $\Sym^kH^2$.
Equation (\ref{initialpolar}) gives now a redefinition of this form by means of a bilinear form on $H^2$. So we liberate ourselves from the initial setting and take the right hand side of (\ref{initialpolar}) as a general recipe to construct a symmetric bilinear form $\bra\ ,\;\ket$ on $\Sym^kV$ from a symmetric bilinear form on an appropriate space $V$. This is carried out in Section~\ref{symSection}. Our main result, Theorem~\ref{maintheorem}, gives a formula for the determinant of the Gram matrix of $\bra\ ,\;\ket$. 

If $V$ is a real vector space, then there is a notable description in terms of an analytic integral given in Prop.~\ref{intequiv}: After some simplifications this amounts to integrating homogeneous polynomials over a sphere. Essentially, we have:
$$ \bra f,g\ket = \int_{\S^d} f(\omega)g(\omega)d\omega $$
This is very comfortable, since it allows to use the whole bunch of techniques from calculus to investigate the algebraic properties of our construction.
Since we are interested in the determinant of the Gram matrix of $\bra\ ,\;\ket$
and for computing determinants it is good to have diagonal matrices, we look for polynomials that are mutually orthogonal on the sphere. The theory of orthogonal polynomials is well developped, and a basis of such polynomials is given by spherical harmonics, see Remark~\ref{sphericalharmonics}. But as spherical harmonics are not suitable for our determinant problem, we construct a different (and slightly simpler)
basis of homogeneous polynomials that are orthogonal on the sphere in Section~\ref{polynomialSection}.

After doing that, we come back to our starting point and apply our results to Hyperk\"ahler manifolds. The bilinear form on $\Sym^kH^2$ allows us to compare $\Sym^kH^2$ with $H^{2k}$. We give some results on torsion factors of the quotient $\frac{H^{2k}}{\Sym^kH^2}$ in Section \ref{hyper}, similar to those the author studied in \cite{Kapfer}.

Sections \ref{symSection}, \ref{polynomialSection} and \ref{hyper} treat rather different aspects and can be read independently.

\section{Terminology and helper formulas} \label{boring}
In this section 
we give a few standard definitions and recall some facts on elementary calculus and lattice theory. We also mention technical formulas needed for our proofs.
\subsection{Combinatorial formulas}
\begin{definition}\label{multiindex}
For a multi-index $\alpha=(\alpha_0,\ldots,\alpha_d)$ of length $\len(\alpha)\coloneqq d+1$ we define: $x^\alpha \coloneqq x_0^{\alpha_0}\ldots x_d^{\alpha_d}$. The degree is defined by $|\alpha |\coloneqq\sum\alpha_i$, the factorial is $\alpha! \coloneqq \prod \alpha_i!$. Further, we set
$\alpha'\coloneqq(\alpha_0,\ldots,\alpha_{d-1})$. We introduce the lexicographical ordering on multi-indices: $\alpha < \beta$ iff $\alpha_d < \beta_d$ or $(\alpha_d=\beta_d) \wedge (\alpha'<\beta')$.
\end{definition}
\begin{definition}
The binomial coefficient for nonnegative integers $k$ and arbitrary $z$ is defined as:
$\binom{z}{k} \coloneqq \frac{z(z-1)\ldots(z-k+1)}{k!}$. Thus we have $\binom{-z}{k}=(-1)^k\binom{z+k-1}{k}$. For negative $k$ we set $\binom{z}{k}\coloneqq 0$.
\end{definition}
We introduce the difference operator $\Delta f(n) \coloneqq f(n+1)-f(n)$. It has the following properties similar to the differential operator:
\begin{align}
 \sum_{i=0}^n \Delta(f) & = f\,\Big|_0^{n+1}= f(n+1)-f(0) && \text{(telescoping sum)} \\
 \Delta (fg)(n) &= f(n+1) \Delta g(n) + g(n)\Delta f(n) &&\text{(product rule)} \\
\label{sumbyparts}  \sum_{i=0}^n g(i)\Delta f(i) & = (fg)\Big|_0^{n+1} - \sum_{i=0}^n f(i+1)\Delta g(i) && \text{(summation by parts)}
\end{align}
This often applies to the binomial coefficient, since we have: 
\begin{equation} \label{binomdiff}
\textstyle \Delta \binom{n}{k}=\binom{n+1}{k}-\binom{n}{k} = \binom{n}{k-1}.
\end{equation}

Let $K$ be a commutative ring and let $r_{d,k} = \rank \left(\Sym^k K^{d+1} \right)$ be the rank of the symmetric power of a free $K$-module of rank $d\!+\!1$. 
Because we have the decomposition $\Sym^k\! K^{d+1} \cong \Sym^k\! K^d \oplus \,(\Sym^{k-1}\! K^{d+1})\!\otimes\! K$, we obtain the recurrence $r_{d,k} = r_{d-1,k} +r_{d,k-1}$. So we deduce:
\begin{equation} \label{binomcount} \textstyle
\binom{k+d}{d} =\binom{k+d}{k} = r_{d,k}=\rank \left(\Sym^k K^{d+1} \right) = \text{card}\big(\{|\alpha| =k\}\big).
\end{equation}

The following identity for integers $d,k\geq 0$ is proven by induction over $k$:
\begin{align} \label{facprod1}
\prod_{j=0}^k (k-j)!^{\binom{j+d-1}{d-1}} &= \ \prod_{i=1}^k i^{\binom{k-i+d}{d}},
\end{align}
where the induction step $k\rightarrow k+1$ produces a factor $\prod\limits_{i=1}^{k+1} i^{\binom{k-i+d}{d-1} }$ on both sides.

We will also need the identity:
\begin{equation} \label{evensum}
\sum_{\substack{i=1\\i\text{ even}}}^{2k+d+1}\textstyle \binom{k-i+d}{d-1} = \left\{ 
 \begin{array}{*2{c}p{5cm}}0 &\text{if }d\text{ is even}, \vspace{1mm}\\
 \binom{k+d}{d} &\text{if }d\text{ is odd},
\end{array}\right.
\end{equation}
which is proven by splitting the sum into:
$$
\sum_{\substack{i=1\\i\text{ even}}}^{k+1} \textstyle\binom{k-i+d}{d-1} + \displaystyle\sum_{\substack{i=k+d+1\\i\text{ even}}}^{2k+d+1}\textstyle \binom{k-i+d}{d-1}
= \displaystyle\sum_{\substack{i=1\\k-i\text{ even}}}^{k+1}\!\! \textstyle\binom{i+d-2}{d-1} +(-1)^{d-1} \!\!\!\!
 \displaystyle\sum_{\substack{i=1\\k+d+i\text{ even}}}^{k+1} \!\!\!\! \textstyle\binom{i+d-2}{d-1} .
$$

\begin{definition}\label{doublefactorial}
We define the double factorial for $ n\geq -1$ by 
$$n!! \,\coloneqq \prod_{i=0}^{\left\lfloor\!\frac{n-1}{2}\!\right\rfloor }(n-2i)=n(n-2)(n-4)\ldots $$
Clearly, $(n-1)!!\,n!! = n!$ and $(2n)!! = 2^n n!$.
\end{definition}
\begin{proposition} \label{partitioncount}
The number of partitions of the set $\{1,\ldots,2k\}$ into pairs equals $(2k-1)!! = \frac{(2k)!}{2^kk!}$.
\end{proposition}
\begin{proof}
Given such a partition, look at the pair that contains the element $1$. There are $2k-1$ possible partners for this element; removing the pair leaves a partition of a set of cardinality $(2k-2)$ into pairs. Then proceed by induction.
\end{proof}
\begin{corollary} \label{multipairs}
Let $D_1,\ldots,D_n$ be disjoint finite sets with $|D_i|=\alpha_i$. Then the number of partitions of the set $D = D_1\cup\ldots\cup D_n$ into pairs, such that the elements of every pair come from the same $D_i$, is equal to $\prod_i (\alpha_i -1)!!$ if all $\alpha_i$ are even and $0$ otherwise.
\end{corollary}

\subsection{Formulas from Calculus}
Denote $\Gamma(t) \coloneqq \int_{0}^{\infty}r^{t-1}e^{-r}dr$ the gamma function. It satisfies:
\begin{align}
\label{ffgamma}
n!&=\Gamma(n+1),\qquad (2n-1)!!\sqrt{\pi} =2^{n}\Gamma\left(n+\tfrac{1}{2}\right), \\
\label{doublegamma}
n!&\sqrt{\pi}=2^{n}\Gamma\left(\tfrac{n}{2}+1\right)\Gamma\left(\tfrac{n+1}{2}\right) ,\\
\int_0^\infty& r^se^{-\frac{1}{2}r^2} dr = 2^{\frac{s-1}{2}}\Gamma\left(\tfrac{s+1}{2}\right).
\end{align}
It follows, that:
\begin{align}    \label{monoint}
 \int_{\R^{d+1}}x^\alpha x^\beta & e^{-\frac{1}{2}\|x\|^2} dx = \ \ 
\displaystyle \prod_{i=0}^d \int_{-\infty}^\infty x_i^{\alpha_i+\beta_i} e^{-\frac{1}{2}x_i^2} dx_i \\
 &= \left\{
\begin{array}{*2{l}p{5cm}}
(2 \pi)^{\frac{d+1}{2}}\prod_{i=0}^d (\alpha_i+\beta_i-1)!! &\text{if all }\alpha_i+\beta_i\text{ are even}, \vspace{0.2cm} \\ 
 0 &\text{otherwise}.
\end{array}
  \right.\nonumber
\end{align}
The reader may also consult \cite{Folland} for that kind of calculus. In particular, \cite[Eq.~(4)]{Folland} yields:
\begin{lemma}\label{homosphere}
Let $f:\R^{d+1}\rightarrow\R$ be a continuous homogeneous function of degree $k$, that is $f(sx) = s^kf(x)\; \forall s\in\!\R$. Then, using polar coordinates $(r,\omega) = (\|x\|,\frac{x}{\|x\|})$:
\begin{align*}
\int_{\R^{d+1}}f(x)e^{-\frac{1}{2}\|x\|^2} dx &= \int_{\S^d}\!\int_0^\infty\! f(r\omega) r^d e^{-\frac{1}{2}r^2}dr d\omega \\
&= 2^{\frac{k+d-1}{2}}\Gamma\!\left(\tfrac{k+d+1}{2}\right)\int_{\S^d}f(\omega)d\omega .
\end{align*}
\end{lemma}

\subsection{Lattices}\label{latticeSubsection}A reference for this subsection is Chapter 8.2.1 of \cite{Dolgachev}. By a lattice $L$ we mean a free $\Z$--module of finite rank, equipped with a non--degenerate, integer--valued symmetric bilinear form $\left<\ ,\;\right>$. By a homomorphism or embedding $L\subset M$ of lattices we mean a map $:L\rightarrow M$ that preserves the bilinear forms on $L$ and $M$ respectively. It is automatically injective. We always have the injection of a lattice $L$ into its dual space $L^*\coloneqq \Hom(L,\Z)$, given by $x \mapsto \left<x,\ \right>$. A lattice is called unimodular, if this injection is an isomorphism, \ie if it is surjective. By tensoring with $\Q$, we can interpret $L$ as well as $L^*$ as a discrete subset of the $\Q$--vector space $L\otimes \Q$. Note that this gives a kind of lattice structure to $L^*$, too, but the symmetric bilinear form on $L^*$ may now take rational coefficients.

If $L\subset M$ is an embedding of lattices of the same rank, then the index $|M\DP L|$ of $L$ in $M$ is defined as the order of the finite group $M/L$.
There is a chain of embeddings $L\subset M \subset M^* \subset L^*$ with $|L^*\DP M^*| =|M\DP L| $.

The quotient $L^*/L$ is called the discriminant group. The index of $L$ in $L^*$ is called $\discr L$, the discriminant of $L$.
Choosing a basis $(x_i)_i$ of $L$, we may express $\discr L$ as the absolute value of the determinant of the so--called Gram matrix $G$ of $L$, which is defined by $G_{ij}\coloneqq \left<x_i,x_j\right>$. $L$ is unimodular, iff $\det G =\pm 1$.

\begin{proposition} \label{squareDiscr}Let $M$ be a unimodular lattice. Let $L\subset M$ be a sublattice of the same rank. Then $|M\DP L|$ equals $\sqrt{\discr L}$.
\end{proposition}
\begin{proof}
Since $M$ is unimodular, $|L^*\DP M|=|L^*\DP M^*| =|M\DP L| $ and therefore $|L^*\DP L| = |L^*\DP M||M\DP L|  = |M\DP L|^2$.
\end{proof}
An embedding $L\subset M$ is called primitive, if the quotient $M/L$ is free. We denote by $L^\perp$ the orthogonal complement of $L$ within $M$. Since an orthogonal complement is always primitive, the double orthogonal complement $ L^{\perp\perp}$ is a primitively embedded overlattice of $L$. It is clear that $\discr( L^{\perp\perp})$ divides $\discr L$. 
\begin{proposition}\label{TorsionQuotient} Let $L\subset M$ be an embedding of lattices. Then the order of the torsion part of $M/L$ 
divides $\discr L$.
\end{proposition}
\begin{proof}
The torsion part is the index of $M/( L^{\perp\perp})$ in $M/L$. But this is equal to $|L^{\perp\perp}\DP L| = |L^* \DP (L^{\perp\perp})^*|$ and clearly divides $|L^*\DP L|$.
\end{proof}
\begin{proposition}\label{discrOrthPrim}
Let $M$ be unimodular. Let $L\subset M$ be a primitive embedding. Then $\discr L = \discr L^\perp$.
\end{proposition}
\begin{proof}
Consider the orthogonal projection $ : M\otimes\Q \rightarrow L\otimes \Q$. Its restriction to $M$ has kernel equal to $L^\perp$ and image in $L^*$. Hence we have an embedding of lattices $M/L^\perp \subset L^*$. Quotienting by $L$, we get an injective map $: M/(L\oplus L^\perp) \rightarrow L^*/L$. 
Now by Proposition~\ref{squareDiscr}, $\sqrt{\discr(L) \discr (L^\perp)} =|M \DP (L\oplus L^\perp)| \leq |L^*\DP L| = \discr L$. So we get $\discr L^\perp \leq \discr L$. Exchanging the roles of $L=L^{\perp\perp}$ and $L^\perp$ gives the inequality in the opposite direction.
\end{proof}
\begin{corollary}\label{latticeCor}
Let $L\subset M$ be an embedding of lattices with unimodular $M$. Let $n$ be the order of the torsion part of $M/L$. Then $\discr L^\perp =\discr L^{\perp\perp} = \frac{1}{n^2}\discr L$.
\end{corollary}

\section{Symmetric Bilinear Forms on Symmetric Powers} \label{symSection}
Let $V$ be a vector space (or a free module) over a field (resp.~a commutative ring) $K$ of rank $d+1$ with basis $\{x_0,\ldots,x_{d}\}$, equipped with a symmetric bilinear form $\left<\,\ ,\ \right>: V\times V \rightarrow K$. We will freely identify the symmetric power $\Sym^kV$ with the space $K[x_0,\ldots,x_d]_k$ of homogeneous polynomials of degree $k$. 

There are at least two possibilities to define an induced bilinear form on $\Sym^kV$. We will use the following
\begin{definition} \label{formdef} On the basis $\{x_{n_1}\ldots x_{n_k}\;|\;0\leq n_1\leq\ldots\leq n_k\leq d\}$ of $\Sym^kV$, we define a symmetric bilinear form $\bra\ \,,\ \ket$ by: 
\begin{equation}
\label{mydef}
\bra x_{n_1}\ldots x_{n_k}\,,\,x_{n_{k+1}}\ldots x_{n_{2k}} \ket \coloneqq \sum_{\mathcal{P}} \prod_{\{i,j\}\in\mathcal{P}} \left<x_{n_i},x_{n_j}\right>,
\end{equation}
where the sum is over all partitions $\mathcal{P}$ of $\{1,\ldots,2k\}$ into pairs.
\end{definition}

We emphasize that this is not the only possibility. One could alternatively define
\begin{equation}\label{Garr}
\left(\!\left( x_{n_1}\ldots x_{n_k}\,,\, x_{m_1}\ldots x_{m_k}\right)\!\right) \coloneqq \sum_\sigma  
\prod_{i=1}^k \left< x_{n_i},x_{m_{\sigma(i)}}\right>,
\end{equation}
the sum being over all permutations $\sigma$ of $\{1,\ldots,k\}$, as studied by McGarraghy in \cite{McGarr}. However, this is a different construction that doesn't match the situation described in the introduction. We will \emph{not} consider it here.

If $U\in O(V)$ is an orthogonal transformation, then the induced diagonal action of $U^{\otimes k}$ on $\Sym^kV$ is orthogonal in both cases. This shows that the values of $\bra\ ,\;\ket$ and $\left(\!\left(\ ,\;\right)\!\right)$ are independent of the choice of the base of $V$ up to orthogonal transformation. 
\begin{example}
To contrast the two definitions, observe that in the case $k=2$
\begin{align}
\bra ab,cd\ket &= \left<a,c\right>\left<b,d\right>+\left<a,d\right>\left<b,c\right> + \left<a,b\right>\left<c,d\right>, \\
\left(\!\left( ab,cd \right)\!\right) &= \left<a,c\right>\left<b,d\right>+\left<a,d\right>\left<b,c\right>.
\end{align}
\end{example}
\begin{remark}
Note that (\ref{Garr}) does not require symmetry of the bilinear form $\left<\,\ ,\ \right>$ on $V$. Indeed, the definition would also be valid for an arbitrary bilinear form~$: V\times W \rightarrow K$, yielding a bilinear form $:\Sym^kV\times\Sym^kW\rightarrow K$. On the other hand, if the form on $V$ is not symmetric, then (\ref{mydef}) is not well-defined.
\end{remark}
\begin{remark}
The defining equation (\ref{mydef}) works equally well, if the two arguments have different degree. So we can easily extend our definition to a symmetric bilinear form~$\bra\ \,,\ \ket:\Sym^*V\times\Sym^*V \rightarrow K$. Then we have: $\bra a,bc\ket =\bra ab,c\ket$. Note that $\Sym^kV$ is in general not orthogonal to $\Sym^lV$ unless $k-l$ is an odd number.
\end{remark}
We wish to investigate some properties of this construction. Let $G$ be the Gram matrix of $\left< \ ,\;\right>$, \ie $G_{ij} = \left<x_i,x_j\right>$ and
let $\G$ be the Gram matrix of $\bra\ ,\;\ket$. We use multi-index notation, cf.~Definition~\ref{multiindex}.
\begin{proposition} \label{intequiv}Assume $K=\R$ and $G$ is positive definite, so its inverse $G^{-1} $ exists. Then $\bra\ ,\;\ket$ takes an analytic integral form:
\begin{equation*}
\bra x^\alpha, x^\beta \ket = \frac{1}{c}\int_{\R^{d+1}} x^\alpha x^\beta d\mu(x),
\end{equation*}
where the integration measure is $d\mu(x) = \exp\left(-\frac{1}{2}\sum_{i,j} G^{-1}_{ij}x_ix_j\right)dx$ and the normalization constant is $c=\int_{\R^{n+1}} d\mu(x)=\sqrt{(2\pi)^{d+1}\det G}$.
\end{proposition}
\begin{proof} Note that we need positive definiteness of $G$ to make the integral converge. We make use of the content in Section \ref{boring}.
First, observe that both sides of the equation are invariant under orthogonal transformations of the base space $\R^{d+1}$. We may therefore assume that $G= \text{diag}\left(a_0,\ldots,a_d\right)$ is a diagonal matrix. Then the integral splits nicely:
\begin{align*}
\frac{1}{c}\int_{\R^{d+1}}& x^\alpha x^\beta d\mu(x)= \frac{1}{c}\prod_{i=0}^d \int_{-\infty}^\infty x_i^{\alpha_i+\beta_i} e^{-\frac{1}{2a_i}x_i^2}dx_i \\=&\: 
\frac{1}{c}\prod_{i=0}^d a_i^{\frac{\alpha_i+\beta_i+1}{2}}\int_{-\infty}^\infty x^{\alpha_i+\beta_i} e^{-\frac{1}{2}x^2}dx\\
\myeq{(\ref{monoint})}\  &\left\{
\begin{array}{*2{l}p{5cm}}\displaystyle \prod_{i=0}^d a_i^{\frac{\alpha_i+\beta_i}{2}}(\alpha_i+\beta_i-1)!! &\text{if all }\alpha_i+\beta_i\text{ are even}, \vspace{0.2cm} \\ 
 0 &\text{otherwise}.
\end{array}
 \right.
\end{align*}
On the other hand, if $G$ is diagonal, then every partition into pairs in Equation (\ref{mydef}) that contains a pair of two different numbers will not contribute to the sum. Corollary~\ref{multipairs} shows then, that we get the same formula for $\bra x^\alpha ,x^\beta\ket$.
\end{proof}

The next theorem gives a formula for the determinant of $\G$. This is of particular interest when $K=\Z$, because in this case we are in the setting of lattice theory, and $\left|\det \G\right|$ is the discriminant of the lattice $\Sym^kV$.
\begin{theorem} \label{maintheorem}
The determinant of the Gram matrix $\G$ of $\bra\ ,\;\ket$, the induced bilinear form on $\Sym^kV,\ \rank V = d\!+\!1$, is:
\begin{equation}\label{maintheorem1}
\det(\G)= \det(G)^{\binom{d+k}{d+1}}\,\theta_{d,k}
\end{equation}
where $\theta_{d,k}$ is a combinatorial factor given by:
\begin{equation} \label{thetaDef}
\theta_{d,k} = \left\{
 \begin{array}{*2{l}p{5cm}}
 \displaystyle\prod_{i=1}^k i^{\binom{k-i+d}{d}d}\prod_{\substack{i=1 \\ i\ \text{odd}\\\ }}^{2k+d-1}i^{\binom{k-i+d}{d}} &\text{if }d\text{ is even}, \\
 \displaystyle\prod_{i=1}^k i^{\binom{k-i+d}{d}d}\prod_{i=1}^{k+\frac{d-1}{2}} i^{\binom{k-i+d}{d} - \binom{k-2i+d}{d}} &\text{if }d\text{ is odd}.
\end{array}
\right.
\end{equation}
\end{theorem}
\begin{remark} If $d$ or $k$ is small, this simplifies as follows:
\begin{gather*}
\theta_{d,0}=\theta_{d,1} =1,\qquad \theta_{d,2} = 2^{d}(d+3), \\
\theta_{0,k} = (2k-1)!!, \qquad \theta_{1,k} = (k!)^{k+1}.
\end{gather*}
\end{remark}
\begin{proof}
We prove the theorem in three steps.
Let us first consider the case when $V$ is a vector space over $\R$ and $G$ is positive definite. We further reduce this to the special case when $G$ is the identity matrix. That is the essential difficulty of the proof, which we will treat in Section~\ref{hsection}.

Any orthogonal transformation $U\in O(V)$ induces a transformation $U^{\otimes k} \in O(\Sym^k V)$ and thus doesn't affect determinants. Since over $\R$, every symmetric matrix can be diagonalized by applying an orthogonal coordinate change, we may assume that $G=\text{diag}\left(a_0,\ldots,a_d\right)$ is a diagonal matrix. 
Let us check, what happens if we apply a coordinate transformation $x\mapsto\tilde x$ that changes the last coordinate by $\tilde{x}_d = \gamma x_d$ and leaves the other coordinates invariant. Let $\tilde{G}$ and $\tilde{\G}$ be the Gram matrices corresponding to the new coordinates.
We clearly have: $\tilde{x}^\alpha = \gamma^{\alpha_d} x^\alpha$. Extracting the factor $\gamma$ from the Leibniz determinant formula, which is of the form $\det \tilde{\G}=\sum\limits_\sigma\pm\prod\limits_{|\alpha|=k} \bra \tilde{x}^{\alpha},\tilde{x}^{\sigma(\alpha)}\ket=\det \G\prod\limits_{|\alpha|=k} \gamma^{2\alpha_d}$, we get: 
\vspace{-2mm}
$$
\frac{\det \tilde{\G}}{\det\G} = \prod_{|\alpha|=k}\gamma^{2\alpha_d} = \prod_{i=0}^k\ \prod_{|\alpha '|=k-i} \gamma^{2i} \; \myeq{(\ref{binomcount})}\;\prod_{i=0}^k\gamma^{2i\binom{k-i+d-1}{d-1}} \; \myeq{(\ref{sumbyparts})} \; \gamma^{2\binom{d+k}{d+1}}.
$$
Now we apply successively coordinate transformations that map $x_i$ to $\frac{x_i}{\sqrt{a_i}}$. We get a factor $(a_0\ldots a_d)^{\binom{d+k}{d+1}} = \det G^{\binom{d+k}{d+1}}$ and we are left with an identity Gram matrix. The statement follows from Theorem~\ref{thetaCor}.

As a second step, still working over $\R$, we show that we can drop the condition that $G$ is positive definite. To see this, let $Q\subset \R^{(d+1)\times(d+1)}$ be the subspace of real symmetric square matrices of size $d+1$. Our formula (\ref{maintheorem1}) depends polynomially on the entries of $G$. The subset $R\subset Q$ of all matrices $G\in Q$ that satisfy (\ref{maintheorem1}) is therefore Zariski-closed. But on the other hand, the positive definite matrices form a nonempty subset $P\subset Q$ which is open in the analytic topology. So if $P\subset R$, then necessarily $R=Q$.

Finally, matrices with integer entries form a subset of real matrices. So (\ref{maintheorem1}) holds also for free $\Z$-modules $V$. But (\ref{maintheorem1}) is an identity living in $\Z[G_{ij}]$, so it holds true over any commutative ring $K$, simply by tensoring with $K$.
\end{proof}

\section{Homogeneous Orthogonal Polynomials on the sphere} \label{polynomialSection}
In this section we will construct a basis for the space of homogeneous polynomials of degree $k$ in $d+1$ variables, $\R[x_0,\ldots,x_d]_k$, that is orthogonal with respect to the bilinear form given by
\begin{equation} \label{spherebracket}
 \bra f,g\ket = \int_{\R^{d+1}}f(x)g(x) d\mu(x),
\end{equation}
where the measure is $d\mu(x) = (2\pi)^{-\frac{d+1}{2}}e^{-\frac{1}{2}\|x\|^2}dx$. In order to do this, we wish to apply the Gram-Schmidt process to the (lexicographically ordered) monomial basis $(x^\alpha)_{|\alpha |=k}$. Our result is stated in Subsection \ref{hsection}. 
\begin{remark}
Although the above definition of $\bra\ ,\;\ket$ doesn't mention the sphere, in view of Lemma \ref{homosphere}, we could equivalently consider the integral:
$$
 \bra f,g\ket = c_{d,k} \int_{\S^{d}}f(\omega)g(\omega) d\omega, \qquad c_{d,k} = 2^{\frac{k}{2}-1}\pi^{-\frac{d+1}{2}}\Gamma\!\left(\tfrac{k+d+1}{2}\right).
$$ 
This is the reason why we speak of polynomials orthogonal on the sphere. However, we prefer to integrate over $\R^{d+1}$, since this avoids the unwanted constant $c_{d,k}$.
\end{remark}

\begin{remark}\label{hermite}
We stress that this equivalency really depends on the homogeneity. Denote $\bra f ,g\ket_{\R^{d+1}} = \bra f,g\ket$ and $\bra f ,g\ket_{\S^{d}}=\int_{\S^{d}} f g$ for a moment and let us look at what happens if we drop the homogeneity constraint. Since $c_{d,k}$ depends on $k$, $\bra\ ,\;\ket_{\R^{d+1}}$ and $\bra\ ,\;\ket_{\S^d}$ aren't equivalent anymore. A basis of $\R[x_0,\ldots,x_d]$, consisting of $\bra f ,g\ket_{\R^{d+1}}$-orthogonal polynomials is given by products $H_{\alpha_0}\!(x_0)\ldots H_{\alpha_d}\!(x_d)$ of Hermite polynomials in one variable, see also \cite[Sect.~2.3.4]{Dunkl}. On the other hand, the form $\bra\ ,\;\ket_{\S^d}$ becomes degenerate on $\R[x_0,\ldots,x_d]$, because integration on the sphere can't distinguish between $1$ and the square radius $\|x\|^2$. 
\end{remark}

\begin{remark}\label{sphericalharmonics}
A $\bra\ ,\;\ket$-orthogonal basis of homogeneous polynomials which we won't consider here is given by spherical harmonics. Let $\Har^{d+1}_k \subset \R[x_0,\ldots,x_d]_k$ be the subspace of harmonic polynomials. By Theorem~1.3 and Proposition~1.4 of \cite{Dai}, there is an orthogonal decomposition
$$
\R[x_0,\ldots,x_d]_k = \Har^{d+1}_k \oplus r^2 \Har^{d+1}_{k-2}\oplus r^4 \Har^{d+1}_{k-4}\oplus \ldots
$$
where $r^2 = \|x\|^2= x_0^2+\ldots +x_d^2$. Orthogonal bases for each of the $\Har^{d+1}_k$ in turn are constructed in \cite[Sect.~2.2]{Dunkl}. However, the basis one obtains this way has nothing to do with monomials. In particular, the transition matrix between them is not triangular, so they are not related by a Gram-Schmidt process.
\end{remark}

\subsection{Generalities on orthogonal polynomials in one variable}
Given a nondegenerate symmetric bilinear form on the space of polynomials $K[x]$, one may ask for a basis of polynomials $(p_n)_n$ that are mutually orthogonal with respect to that form. To find such a basis, one could start with the monomial basis $(x^n)_n$ and apply some version of the Gram--Schmidt algorithm. The result will be an infinite lower triangular matrix $T$ such that $p_n =\sum_j T_{nj} x^j$. We prefer to normalize such that the diagonal elements of $T$ are equal to $1$. If our bilinear form now depends only on the product of its two arguments, the procedure simplifies as follows:

Let $\L$ be a linear functional such that the induced bilinear form $(f,g)=\L(fg)$ is nondegenerate when restricted to $K[x]_{\leq n}$, the space of polynomials of bounded degree, for all $n\geq 0$. Let $(p_n)_n$ be the associated sequence of monic orthogonal polynomials, \ie the leading term of $p_n(x)$ is $x^n$ and $(p_k,p_{n})=0$ for $k\neq n$. Then we have
\begin{theorem} \cite[Thm.~4.1]{Chihara} There are constants $c_n,\: d_n$ such that
\begin{equation*}
 p_0(x) = 1,\qquad  p_{n+1}(x) = (x-c_n)p_n(x) - d_np_{n-1}(x).
\end{equation*}
\end{theorem}
But also the converse is true:
\begin{theorem}[Favard's theorem] \cite[Thm.~4.4]{Chihara} \label{favard}
Let $(p_n)_n$ be a sequence of polynomials, such that $\deg p_n =n$ and the following three-term recurrence holds:
$$p_0(x) = 1,\qquad  p_{n+1}(x) = (x-c_n)p_n(x) - d_np_{n-1}(x).
$$
Then there exists a unique linear functional $\L$ such that $\L(1)=1$ and $\L(p_kp_{n})=0$ for $k\neq n$. 
\end{theorem}
\begin{theorem}\cite[Thm.~4.2]{Chihara} \label{generalLnorm}
Under the conditions of the above theorems, we have for $n\geq 1$: $$\L(p_n^2) = d_n\L(p_{n-1}^2).$$
\end{theorem}
\begin{remark} \label{finiteFarvard}Since we shall deal with finite polynomial families, we need a little modification of Favard's theorem:
If $(p_n)_{n\leq N}$ is a finite sequence that satisfies a three-term recurrence as above, then we can always extend it to an infinite sequence by choosing arbitrary constants $c_n$, $d_n$ for $n\geq N$. But for every such extension, the resulting functional $\L$ from Favard's theorem will satisfy $\L(1) =1$ and $\L(p_n)=\L(p_np_0)=0$ for $n\geq 1$. So $\L$ will always be uniquely determined on $K[x]_{\leq N}$, the space of degree-bounded polynomials.
\end{remark}

\subsection{A polynomial family} 
Our construction of homogeneous polynomials orthogonal on the sphere is formally similar to the definition of spherical harmonics, see \cite[p.~35]{Dunkl}. Those are defined by recursion over the number of variables, as products of Chebychev and Gegenbauer polynomials. Inspired by that procedure, we introduce the following polynomial family in lieu thereof:
\begin{definition} Let $n,\ m$ be integers with $0\leq 2n\leq m+1$, a condition that we always will assume silently. We define monic polynomials $\p{n}{m}$ of degree $n$ with rational coefficients:
\begin{equation*}
\p{n}{m}(x) \coloneqq \sum_{\substack{j=0\\ n-j\ \text{even}}}^n (-1)^{\frac{n-j}{2}} \frac{n!\,(m-2n)!!}{j!\,(m-n-j)!!\,(n-j)!!}\:x^j.
\end{equation*}
\end{definition}
\begin{remark}
As Yuan Xu pointed out to the author, this can be written in terms of the hypergeometric function, namely $\p{n}{m}(x) = x^n\vphantom{}_2F_1\!\left( \genfrac{}{}{0pt}{}{-\frac{n}{2},\,\frac{1-n}{2}}{\frac{m}{2}-n+1} ; -\frac{1}{x^2}\right)$. To see this, first change summation from $j$ to $n-j$, so that the sum is over $j\! =$ even, then set $j\! =\!2i$ and rewrite the sum in the notation of the rising Pochhammer symbol $(a)_n = a (a\!+\!1)\ldots(a\!+\!n\!-\!1)$. Comparing this formula with \cite[Prop.~1.4.11]{Dunkl}, it follows that the $\p{n}{m}$ are in fact a variant of the Gegenbauer polynomials $C_n^\lambda$. More precisely, $\p{n}{m}(x)$ is a multiple of $C^{-\frac{m}{2}}_n\!\!\left(\sqrt{-1}\,x\right)$, where the factor is chosen such that the polynomial becomes monic.
\end{remark}

\begin{lemma} \label{trigonometric}
For $n\geq 1$, we have a trigonometric differential relation:
\begin{equation*}
\frac{d}{d\omega} \Big[\p{n-1}{m-2}\big(\tan(\omega)\big)\cos(\omega)^{m-1} \Big]= (n-m)\, \p{n}{m}\big(\tan(\omega)\big)\cos(\omega)^{m-1}.
\end{equation*} 
\end{lemma}
\begin{proof} This is straightforward.
Firstly, we calculate $ \frac{d}{d\omega}\! \left[\sin(\omega)^j \cos(\omega)^{m-j-1}\right] = j \sin(\omega)^{j-1} \cos(\omega)^{m-j}-(m\!-\!j\!-\!1)\sin(\omega)^{j+1} \cos(\omega)^{m-j-2}$, and so
\begin{align*}
\frac{d}{d\omega} &\Big[\p{n-1}{m-2}\big(\tan(\omega)\big)\cos(\omega)^{m-1} \Big] 
\\&=\!\! \textstyle \sum\limits_{\substack{j=0\\ n-j\ \text{odd}}}^{n-1} \!\!\!\frac{(-1)^{\frac{n-j-1}{2}}  (n-1)!\,(m-2n)!!}{j!\,(m-n-j-1)!!\,(n-j-1)!!}\frac{d}{d\omega}\! \left[\sin(\omega)^j \cos(\omega)^{m-j-1}\right]
\\ &=\!\!\textstyle\sum\limits_{\substack{j=0\\ n-j\ \text{even}}}^{n-2} \!\!\!\frac{(-1)^{\frac{n-j-2}{2}}(n-1)!\,(m-2n)!!}{j!\,(m-n-j-2)!!\,(n-j-2)!!}\sin(\omega)^{j} \cos(\omega)^{m-j-1} \\[-3mm]
&\hspace{3cm}-\textstyle\sum\limits_{\substack{j=1\\ n-j\ \text{even}}}^{n}\!\!\! \frac{(-1)^{\frac{n-j}{2}}(n-1)!\,(m-2n)!!\,(m-j)}{(j-1)!\,(m-n-j)!!\,(n-j)!!}\sin(\omega)^{j} \cos(\omega)^{m-j-1} \\
&=\!\!\textstyle\sum\limits_{\substack{j=0\\ n-j\ \text{even}}}^{n} \!\!\!\frac{(-1)^{\frac{n-j}{2}}(n-1)!\,(m-2n)!!}{j!\,(m-n-j)!!\,(n-j)!!}\underbrace{\big[(j\!-\!n)(m\!-\!n\!-\!j)-j(m\!-\!j)\big] }_{=n(n-m)}\tan(\omega)^{j} \cos(\omega)^{m-1}
\\ &= \textstyle(n-m)\, \p{n}{m}\big(\tan(\omega)\big)\cos(\omega)^{m-1}. \qedhere
\end{align*} 
\end{proof}

Our next goal is to show that the $\p{n}{m}$, for fixed $m$, form a set of orthogonal polynomials in the sense of the above subsection. In order to apply Favard's theorem, we claim:
\begin{proposition} \label{threeterm} For $0\leq 2n\leq m-1$, we have a three-term recurrence:
\begin{equation*}
\p{0}{m}(x) = 1,\qquad \p{1}{m}(x) = x, \qquad \p{n+1}{m}(x) = x\p{n}{m}(x) -d_n^m \p{n-1}{m}(x),
\end{equation*} 
where $d_n^m\coloneqq \frac{n(m-n+1)}{(m-2n)(m-2n+2)}$.
\end{proposition}
\begin{proof}
We start from the right: $ x\p{n}{m}(x) -d_n^m \p{n-1}{m}(x)$ gives
\begin{align*}
 &\textstyle \sum\limits_{\substack{j=1\\ n-j\ \text{odd}}}^{n+1} \!\!\!\frac{(-1)^{\frac{n-j+1}{2}} n!\,(m-2n)!!}{(j-1)!\,(m-n-j+1)!!\,(n-j+1)!!}\: x^j 
\ - \sum\limits_{\substack{j=0\\ n-j\ \text{odd}}}^{n-1} \!\!\!d_n^m\frac{(-1)^{\frac{n-j-1}{2}} (n-1)!\,(m-2n+2)!!}{j!\,(m-n-j+1)!!\,(n-j-1)!!}\: x^j 
\\=&\textstyle \sum\limits_{\substack{j=0\\ n-j\ \text{odd}}}^{n+1} \!\!\!\frac{(-1)^{\frac{n-j+1}{2}} (n+1)!\,(m-2n-2)!!}{j!\,(m-n-j-1)!!\,(n-j+1)!!}\underbrace{\textstyle
\frac{j(m-2n)+(m-n+1)(n-j+1)}{(n+1)(m-n-j+1)}}_{=1} \,x^j = \p{n+1}{m}(x) .\qedhere
\end{align*} 
\end{proof}
The next theorem gives a useful analytic form of the corresponding linear functional.
\begin{theorem} \label{pthm}We define, for $m\geq 1$, a linear functional $\L$ on the vector space of polynomials of degree less than $m$, by setting
\begin{equation}
\L: f \longmapsto \int_0^\infty\!\!\! \int_{-\infty}^\infty z^{m-1}f\left(\frac{y}{z}\right) e^{-\frac{y^2+z^2}{2}} dydz.
\end{equation}
Then the $\p{n}{m}$ form a set of orthogonal polynomials with respect to the induced bilinear form, \ie for $k\neq n,\ k\!+\! n\leq m\!-\!1$, we have $\L(\p{k}{m}\p{n}{m})=0$ and for $2n\leq m\!-\!1$:
\begin{equation} \label{Lpn2}
\L(\p{n}{m}\p{n}{m}) = 2^{\frac{3}{2}m-2n-\frac{1}{2}}  \frac{n!}{(m-n)!}
\Gamma\left(\frac{m}{2}-n\right)\Gamma\left(\frac{m}{2}-n+1\right)\Gamma\left(\frac{m+1}{2}\right).
\end{equation}
\end{theorem}
\begin{proof}
Since $\p{n}{m}$ satisfy the three-term relation in Prop.~\ref{threeterm}, by Favard's theorem and Remark \ref{finiteFarvard}, there exists a unique functional $\L'$ with $\L'(1)=1$, such that the $\p{n}{m}$ form an orthogonal basis with respect to the bilinear form induced by $\L'$. We claim that $\L$ is a scalar multiple of $\L'$.
Since $(\p{n}{m})_{n} $ is a basis of the space of polynomials, we must show that, for $n\geq 1$, $\L(\p{n}{m}) = \L(\p{n}{m}\p{0}{m}) =0$.
Using polar coordinates $(y,z) = (r\cos\omega,r\sin\omega)$ and Lemma \ref{trigonometric}, we get:
\begin{align*}
\L(\p{n}{m}) &= \int_0^\pi\!\!\int_0^\infty\p{n}{m}\!\left(\tfrac{\cos\omega}{\sin\omega}\right)\sin(\omega)^{m-1} r^m e^{-\frac{r^2}{2}}dr d\omega \\
&=\int_0^\infty r^m e^{-\frac{r^2}{2}}dr \int_{\frac{\pi}{2}}^{\frac{3\pi}{2}} (-1)^n \p{n}{m}\big(\tan(\omega)\big)\cos(\omega)^{m-1}d\omega\\
&= 2^{\frac{m-1}{2}}\Gamma\left(\tfrac{m+1}{2}\right)\Big[\tfrac{(-1)^n}{n-m}p_{n-1}^{m-2}\big(\tan(\omega)\big)\cos(\omega)^{m-1} \Big]_{\frac{\pi}{2}}^{\frac{3\pi}{2}} =0,
\end{align*}
while $\L(1) = 2^{\frac{m-1}{2}}\sqrt{\pi}\,\Gamma\!\left(\frac{m}{2}\right)=2^{\frac{3m-1}{2}}  \frac{1}{m!}
\Gamma\left(\frac{m}{2}\right)\Gamma\left(\frac{m}{2}+1\right)\Gamma\left(\frac{m+1}{2}\right)$ by (\ref{monoint}) and (\ref{doublegamma}). To verify that equation (\ref{Lpn2}) holds for $n\geq 1$, too, we must show that the right hand side satisfies the recurrence from Theorem~\ref{generalLnorm}, but this is immediate:
$$ 
\frac{2^{\frac{3}{2}m\!-\!2n\!-\!\frac{1}{2}}  \frac{n!}{(m\!-\!n)!}
\Gamma\!\left(\frac{m}{2}\!-\!n\right)\Gamma\!\left(\frac{m}{2}\!-\!n\!+\!1\right)\Gamma\!\left(\frac{m+1}{2}\right)}{
2^{\frac{3}{2}m\!-\!2n\!+\!\frac{3}{2}}  \frac{(n\!-\!1)!}{(m\!-\!n\!+\!1)!}
\Gamma\!\left(\frac{m}{2}\!-\!n\!+\!1\right)\Gamma\!\left(\frac{m}{2}\!-\!n\!+\!2\right)\Gamma\!\left(\frac{m+1}{2}\right)} = \frac{n(m\!-\!n\!+\!1)}{(m\!-\!2n)(m\!-\!2n\!+\!2)} =d_n^m.
$$
\end{proof}
\begin{corollary} \label{pcor}
 $\L(x^{k}\p{n}{m}) = 0$ for $k < n$ and $\L(x^n\p{n}{m})= \L(\p{n}{m}\p{n}{m})$.
\end{corollary}
\begin{proof} By the theorem, we have $\L(x^{0}\p{n}{m}) = 0$ for $n>0$, so the case $k=0$ holds true.
Now the three-term recurrence from Proposition \ref{threeterm} allows us to inductively conclude that $\L(x^k\p{n}{m}) = \L(x^{k-1}\p{n+1}{m})+d_n^m\L(x^{k-1}\p{n-1}{m}) =0$.
The second assertion, $\L(x^n\p{n}{m})= \L(\p{n}{m}\p{n}{m})$ is trivial in the case $n\leq 1$. For $n\geq 1$, the three-term recurrence yields now $\L(x^n\p{n}{m}) = \L(x^{n-1}\p{n+1}{m})+d_n^m\L(x^{n-1}\p{n-1}{m})=d_n^m\L(x^{n-1}\p{n-1}{m})$, so $\L(x^n\p{n}{m})$ and $\L(\p{n}{m}\p{n}{m})$ (by Theorem \ref{generalLnorm}) satisfy the same recurrence relation and therefore must be equal.
\end{proof}

\subsection{Homogeneous orthogonal polynomials} \label{hsection}
We are now ready to give the desired basis $(h_\alpha)_\alpha$ of homogeneous polynomials that are orthogonal on the sphere. 
\begin{definition} \label{hdef}
For multi-indices $\alpha=(\alpha_0,\ldots,\alpha_d)$ we recursively define homogeneous polynomials $h_\alpha$ of degree $|\alpha |$ by $h_{(\alpha_0)}(x) \coloneqq x_0^{\alpha_0}$ and, for $d\geq 1$,
$$
h_\alpha(x) \coloneqq \p{\alpha_d}{2|\alpha |+d}\left(\frac{x_d}{r}\right) r^{\alpha_d}h_{\alpha'}(x'),
$$
where we have set $r=\sqrt{x_0^2 +\ldots+x_{d-1}^2}=\|x'\|$. Note that the definition of $\p{n}{m}$ implies that $\p{n}{m}(\frac{1}{y})y^n$ is an even polynomial, so all square roots vanish.
\end{definition}
\begin{theorem} Let $\bra\ ,\;\ket$ be defined as in (\ref{spherebracket}).
For all multi-indices $\alpha,\;\beta$ of length $\len(\alpha)=\len(\beta)=d+1$ and degree $|\alpha|=|\beta|$ we have:
\begin{align} \label{hth1}
\bra h_\alpha ,h_\alpha \ket &= \alpha_d!\,\frac{\left(2|\alpha '|\!+\!d\right)!!\,\left(2|\alpha |\!+\!d\!-\!1\right)!!}{\left(|\alpha '|+|\alpha |+d\right)!}\bra h_{\alpha '},h_{\alpha '}\ket ,\\
\label{hth2}
\bra h_\alpha,h_\beta \ket &= 0 \quad\text{for }\alpha\neq\beta,\\
\label{hth3}
\bra x^\alpha, h_\alpha \ket &= \bra h_\alpha ,h_\alpha \ket ,\\
\label{hth4}
\bra x^\alpha,h_\beta \ket &= 0 \quad\text{for }\alpha < \beta\ \text{(see Def.~\ref{multiindex})}.
\end{align}
\end{theorem}
\begin{remark}
This means that the $h_\alpha(x),\ |\alpha|=k$ form an orthogonal basis of $\R[x_0,\ldots,x_d]_k$ that comes from a Gram--Schmidt process applied to the monomials (in lexicographic order). Indeed, equations (\ref{hth1}) and (\ref{hth2}) imply that the $h_\alpha(x)$ are orthogonal, while equations (\ref{hth3}) and (\ref{hth4}) say that the transition matrix $T^{-1}$, defined by $ x^\alpha = \sum_\beta T^{-1}_{\alpha\beta}h_\beta,\ T^{-1}_{\alpha\beta} \coloneqq \frac{\bra x^\alpha,h_\beta\ket}{\bra h_\beta,h_\beta\ket} $ is lower triangular with all diagonal elements equal to $1$.
\end{remark}
\begin{proof}
We begin with equation (\ref{hth1}). The term $\bra h_{\alpha '},h_{\alpha '}\ket$ on the right hand side means of course the form in $d$ instead of $d\!+\!1$ variables. We use polar coordinates on $\R^d$ to compute:
\begin{align*}
 & \sqrt{2\pi}^{d+1}\!\bra h_\alpha ,h_\alpha \ket = \int_{\R^{d+1}}\left[\p{\alpha_d}{2|\alpha|+d}\left(\frac{x_d}{r}\right) r^{\alpha_d}h_{\alpha'}(x)\right]^2 e^{-\frac{1}{2}\|x\|^2} dx \\
 = & \underbrace{ \int_0^\infty\!\!\int_{\R} \left[\p{\alpha_d}{2|\alpha|+d}\left(\frac{x_d}{r}\right)\right]^2 r^{2|\alpha'|+2\alpha_d+d-1} e^{-\frac{r^2+x_d^2}{2} }dx_ddr}_{= \L\left(\left(\p{\alpha_d}{2|\alpha|+d}\right)^2\right)}
 \int_{\S^{d-1}}\left[h_{\alpha'}(\omega)\right]^2 d\omega \\
\myeq{(\ref{Lpn2})}& \  \frac{\alpha_d!\, 2^{2|\alpha '|+|\alpha|+\frac{3}{2}d-\frac{1}{2}} }{\left(|\alpha|+|\alpha'|+d\right)!} { \textstyle
\Gamma\!\left( |\alpha'| \!+\!\frac{d}{2}\right)\Gamma\!\left( |\alpha'|\!+\! \frac{d}{2}\!+\! 1\right) \Gamma\!\left( |\alpha| \!+\! \frac{d+1}{2}\right) }\int_{\S^{d-1}}\!\!\left[h_{\alpha'}(\omega)\right]^2 d\omega \\
\myeq{Lemma~\ref{homosphere}}&\quad\quad \alpha_d!\,2^{|\alpha '|+|\alpha |+d+\frac{1}{2}}\,
\frac{\Gamma\!\left(|\alpha '|\!+\!\frac{d}{2}\!+\!1\right)\Gamma\!\left(|\alpha |\!+\!\tfrac{d+1}{2}\right)}{\left(|\alpha '|+|\alpha |+d\right)!} \int_{\R^d} \left[h_{\alpha'}(x')\right]^2 e^{-\frac{1}{2}\|x'\|^2} dx' \\
\myeq{(\ref{ffgamma})}& \quad\; 
\alpha_d!\,\frac{\left(2|\alpha '|\!+\!d\right)!!\,\left(2|\alpha |\!+\!d\!-\!1\right)!!}{\left(|\alpha '|+|\alpha |+d\right)!}\, \underbrace{\sqrt{2\pi}\int_{\R^d} \left[h_{\alpha'}(x')\right]^2 e^{-\frac{1}{2}\|x'\|^2}  dx'}_{=\sqrt{2\pi}^{d+1}\bra h_{\alpha '},h_{\alpha '}\ket} .
\end{align*}
For the proof of (\ref{hth2}), let $i$ be the highest index where $\alpha$ and $\beta$ differ. Due to the recursive nature of (\ref{hth1}) and the definition of $h_\alpha$, we may assume $i=d$, so $\alpha_d \neq \beta_d$. Then we use the calculation above to see that Theorem~\ref{pthm} now implies the vanishing of the integral. Equations (\ref{hth3}) and (\ref{hth4}) follow analogously from Corollary~\ref{pcor}.
\end{proof}
\begin{theorem}\label{thetaCor} Let the bilinear form $\bra\ ,\;\ket$ on $\R[x_0,\ldots x_d]_k$ be defined as in (\ref{spherebracket}). Let $D(d,k):=\det\limits_{|\alpha|,|\beta|=k}\bra x^\alpha ,x^\beta \ket $ be the determinant of its Gram matrix. Then:
$$
D(d,k) = \theta_{d,k}
$$
where $\theta_{d,k}$ is defined as in (\ref{thetaDef}).
\end{theorem}
\begin{proof} This is a double induction over $k$ and $d$. First check that $D(d,0) = \theta_{d,0}=1$ and $D(0,k) = \theta_{0,k} =(2k-1)!!$.
From the above theorem it is clear that $D(d,k)=\prod_{|\alpha | =k}\bra h_\alpha ,h_\alpha \ket$. Since $\{ |\alpha| = k \} = \bigcup_{j=0}^k  \{ |\alpha'| = j \} \times \{ k-j \}  $, we have from (\ref{hth1}):
$$
D(d,k) = \prod_{j=0}^k D(d\!-\!1,j) \prod_{|\alpha'| = j} (k-j)!\,\frac{(2j+d)!!\,(2k+d-1)!!}{(j+k+d)!},
$$
hence we get the ratio \vspace{-1mm}
\begin{align*}
R(d,k) \;&\coloneqq\; \frac{D(d,k)}{\prod_{j=0}^k D(d\!-\!1,j) } \;=\; \prod_{j=0}^k \left[(k-j)!\frac{(2j+d)!!\,(2k+d-1)!!}{(j+k+d)!} \right]^{\binom{j+d-1}{d-1}} \\
&\qquad\myeq{(\ref{facprod1})}\; \prod_{j=0}^k\left[ \frac{(2j+d)!!}{(j\!+\!k\!+\!d)!!}\right]^{\binom{j+d-1}{d-1}}(2k\!+\!d\!-\!1)^{\binom{k+d}{d}} \;\prod_{i=1}^k i^{\binom{k-i+d}{d}} .
\end{align*}
We will now show the principal inductive step: $\frac{D(d,k+1)}{D(d,k)D(d-1,k+1)}=\frac{\theta_{d,k+1}}{\theta_{d,k}\theta_{d-1,k+1}}$. The left hand side clearly equals
\begin{align*}
 \frac{R(d,k+1)}{R(d,k)} &= \frac{(2k\!+\!d\!+\!2)!!^{\binom{k+d}{d-1}}\,(2k\!+\!d\!+\!1)^{\binom{k+d+1}{d}}\,(2k\!+\!d\!+\!1)!!^{\binom{k+d}{d-1}}}{(2k\!+\!d\!+\!2)!^{\binom{k+d}{d-1}} \prod\limits_{j=0}^k(j\!+\!k\!+\!d\!+\!1)^{\binom{j+d-1}{d-1}}} \prod_{i=1}^{k+1} i^{\binom{k-i+d}{d-1}}\\
&= \frac{(2k+d+1)^{\binom{k+d}{d}}}{\prod\limits_{i=k+d+1}^{2k+d+1}i^{ \binom{k-i+d}{d-1}} }\prod_{i=1}^{k+1} i^{\binom{k-i+d}{d-1}}.
\end{align*}
To simplify the right hand side a little bit, we split $\theta_{d,k} = A(d,k)B(d,k)$ with $A(d,k)\coloneqq\prod_{i=1}\limits^k i^{\binom{k-i+d}{d}d}$, and $B(d,k)$ the complementary factor of $\theta_{d,k}$ depending on the parity of $d$. For $A(d,k)$ we have:
\begin{align*}
\frac{A(d,k+1)}{A(d,k)A(d-1,k+1)} &=
 \prod_{i=1}^{k+1}i^{\binom{k-i+d+1}{d}d-\binom{k-i+d}{d}d -\binom{k-i+d}{d-1}(d-1)} \ 
 \myeq{(\ref{binomdiff})}\ \prod_{i=1}^{k+1} i^{\binom{k-i+d}{d-1}},
\end{align*}
while the other factor $B(d,k)$ gives, for even $d$,
\begin{gather*}
\frac{B(d,k+1)}{B(d,k)B(d-1,k+1)} = \frac{ (2k\!+\!d\!+\!1)^{\binom{-k-1}{d}} \prod\limits_{\substack{i=1 \\ i\ \text{odd} }}^{2k+d+1}i^{\binom{k-i+d+1}{d}-\binom{k-i+d}{d}}  }{\prod\limits_{i=1}^{k+\frac{d}{2}} i^{\binom{k-i+d}{d-1}} \prod\limits_{\substack{i=1 \\ i\ \text{even} }}^{2k+d} \left(\frac{i}{2}\right)^{- \binom{k-i+d}{d-1}} } \\
\myeq{(\ref{binomdiff})}\ \frac{ (2k\!+\!d\!+\!1)^{\binom{k+d}{d}}  \prod\limits_{i=1}^{2k+d+1}i^{\binom{k-i+d}{d-1}}  }{ \prod\limits_{i=1}^{k+\frac{d}{2}} i^{\binom{k-i+d}{d-1}} \prod\limits_{\substack{i=1 \\ i\ \text{even} }}^{2k+d} 2^{ \binom{k-i+d}{d-1}} } 
\ \myeq{(\ref{evensum})} \ \frac{ (2k\!+\!d\!+\!1)^{\binom{k+d}{d}} }{ \prod\limits_{i=k+d+1}^{2k+d+1}i^{ \binom{k-i+d}{d-1}} },
\end{gather*}
but also for odd $d$,
\begin{gather*}
\frac{B(d,k+1)}{B(d,k)B(d-1,k+1)} = \frac{ (k+\tfrac{d+1}{2})^{-\binom{-k-1}{d}} \prod\limits_{i=1}^{k+\frac{d+1}{2}} i^{\binom{k-i+d}{d-1}-\binom{k-2i+d}{d-1}}  }{ \prod\limits_{\substack{i=1 \\ i\ \text{odd} }}^{2k+d+1}i^{\binom{k-i+d}{d-1}} } \\
= \frac{ (k+\tfrac{d+1}{2})^{\binom{k+d}{d}} \prod\limits_{i=1}^{k+\frac{d+1}{2}} i^{\binom{k-i+d}{d-1}}  \prod\limits_{\substack{i=1 \\ i\ \text{even} }}^{2k+d} 2^{ \binom{k-i+d}{d-1}}}{ \prod\limits_{i=1}^{2k+d+1}i^{\binom{k-i+d}{d-1}}} 
\ \myeq{(\ref{evensum})} \ \frac{ (2k\!+\!d\!+\!1)^{\binom{k+d}{d}} }{ \prod\limits_{i=k+d+1}^{2k+d+1}i^{ \binom{k-i+d}{d-1}} }. \qedhere
\end{gather*}\nobreak
\end{proof}

\section{Application in Hyperk\"ahler geometry} \label{hyper}
Let $X$ be a compact Hyperk\"ahler manifold of complex dimension $2k$. The second cohomology group $H^2(X,\Z)$ comes with an integral quadratic form, called the Beauville--Bogomolov form $q_X$, which can be computed by an integration over some cup--product power, see \cite[Subsection~2.3]{OGrady}:
\begin{equation} \label{fujiki}
\int_X \alpha ^{2k} = (2k-1)!!\,c_X q_X(\alpha)^k,\qquad \alpha\in H^2(X,\Z).
\end{equation}
This equation is referred to as the Beauville--Fujiki relation. The constant $c_X\in\Q$ is chosen such that the quadratic form $q_X$ is indivisible and its signum is such that $q_X(\sigma + \bar{\sigma}) > 0$ for a holomorphic two-form $\sigma$ with $\int_X\sigma\bar{\sigma} = 1$. There is an alternative description, as shown in \cite[Chap.~23]{Huybrechts}. Up to a scalar factor $\tilde{c}$, $q_X$ is equal to:
\begin{equation}\label{bb}
 \tilde{c}\,q_X(\alpha) = \frac{k}{2}\int_X \alpha^2 (\sigma\bar{\sigma})^{k-1} + (1-k)\left(\int_X\alpha\,\sigma^{n-1}\bar{\sigma}^{n}\right)\left(\int_X\alpha\,\sigma^{n}\bar{\sigma}^{n-1}\right).
\end{equation}
Now $q_X$, by polarization, gives rise to a symmetric bilinear form $\left<\ ,\;\right>$ on $H^2(X,\Z)$, namely $2\left<\alpha,\beta\right> \coloneqq q_X(\alpha+\beta)-q_X(\alpha) -q_X(\beta)$. On the other hand, from (\ref{fujiki}) one deduces again by polarization, as shown in the introduction, see also \cite[Eq.~3.2.4]{OGrady} that:
\begin{equation}
 \int_X \alpha_1\wedge\ldots\wedge\alpha_{2k} = c_X \bra \alpha_1\ldots\alpha_k\,,\,\alpha_{k+1}\ldots\alpha_{2k}\ket,
\end{equation}
with the induced form $\bra\ ,\;\ket$ on $\Sym^kH^2(X,\Z)$, according to Definition \ref{formdef}. The discriminant of $\Sym^kH^2(X,\Z)$ can be computed with Theorem~\ref{maintheorem}. Since the Poincar\'e pairing $(\beta_1,\beta_2)_X \coloneqq \int_X\beta_1\wedge\beta_2$ gives $H^{2k}(X,\Z)$ the structure of a unimodular lattice, we have got an embedding of lattices:
\begin{equation}
\Big( \Sym^kH^2(X,\Z),\: c_X\!\bra\ ,\;\ket\Big) \longrightarrow \Big(H^{2k}(X,\Z),\;(\ ,\;)_X\Big).
\end{equation}
In general, this embedding is not primitive.

From Theorem \ref{maintheorem} and Proposition~\ref{TorsionQuotient} we now get some interesting corollaries, by looking at the prime factors contained in (\ref{maintheorem1}):
\begin{corollary}
Let $X$ be a compact Hyperk\"ahler manifold of complex dimension $2k$. Denote $b_2$ resp.~$d_2$ the rank and the discriminant of $H^2(X,\Z)$. Define a set of integers $Z$ by
$$
Z \coloneqq \{c_X^{b_2} d_2\} \,\cup\, \{1,\ldots ,k\} \,\cup\, \left\{
\begin{array}{*2{l}p{5cm}}
 \{i \in\Z\,|\,k+b_2 \leq i \leq 2k\!+\!b_2\!-\!2,\; i\; \text{odd}\}   &\text{if }b_2\text{ is odd}, \\
 \{i\in\Z \,|\, \frac{k+b_2}{2}\leq i \leq k\!+\!\frac{b_2}{2}\!-\!1\}   &\text{if }b_2\text{ is even}.
\end{array}
\right.
$$
Then the discriminant of $\Sym^kH^2(X,\Z)$ and hence the torsion part of the quotient
$$
\frac{H^{2k}(X,\Z)}{\Sym^kH^2(X,\Z)}
$$
contain only prime factors that divide at least one of the numbers contained in $Z$.
\end{corollary}
For the known examples of compact Hyperk\"ahler manifolds in higher dimensions, we can use the list given in \cite[Table~1]{OGrady}:
Let $S^{[k]}$ for $k\geq 2$ be the Hilbert scheme of $k$ points on a K3 surface $S$, let $A^{[[k]]}$ be the generalized Kummer variety of a torus $A$ and let $OG_6$ and $OG_{10}$ be the 6-- resp. 10--dimensional O'Grady manifold. Let $X$ be deformation equivalent to one of these. Then we have:

\begin{center}
\begin{tabular}{ p{2cm} p{2cm} p{2cm} p{2cm} p{2cm} }
$X$ 		& $\dim X$	& $b_2$ & $d_2$	& $c_X$ \\ \hline 
$S^{[k]}$ 	& $2k$ 		& $23$ 	& $ 2(k-1)$		& $1$	\\
$A^{[[k]]}$	& $2k$ 		& $7$	& $ 2(k+1)$		& $k+1$ \\
$OG_6$		& $6$		& $8$	& $4$			& $4$ 	\\
$OG_{10}$	& $10$		& $24$	& $3$			& $1$ 	
\end{tabular}
\end{center}
In particular, the torsion part of $\frac{H^{2k}(X,\Z)}{\Sym^kH^2(X,\Z)}$ contains no prime factors bigger than
\begin{itemize}
\item $3$, if $X\defeq OG_6$,
\item $5$, if $X\defeq OG_{10}$.
\end{itemize} 
\begin{remark}
The case $X\defeq S^{[2]}$ was already studied in \cite[Prop.~6.6]{BNS}, using explicit calculations. It is special, because $\Sym^2H^2(S^{[2]},\Z)$ and $H^4(S^{[2]},\Z)$ have the same rank. So Proposition~\ref{squareDiscr}, together with Theorem~\ref{maintheorem} imply that the cardinality of the quotient is precisely $\sqrt{2^{24}\cdot2^{22}\cdot(22+3)} = 2^{23}\cdot5$.
\end{remark}
\begin{remark}
If the exact value of the torsion part of $H^{2k}/(\Sym^kH^2)$ is known, we can say more, by applying Corollary~\ref{latticeCor}. For instance,
in the case $X\defeq S^{[3]}$, Theorem~\ref{maintheorem} says that the discriminant of $\Sym^3H^2$ is equal to $2^{1106}\cdot 3^{92}$. But on the other hand, according to \cite[Prop.~2.4]{Kapfer}, the torsion part of $H^6/(\Sym^3H^2)$ has order $2^{277}\cdot 3^{46}$. So it follows that both the orthogonal complement and the primitive overlattice of $\Sym^3H^2$ must have discriminant $\frac{2^{1106}\cdot 3^{92}}{(2^{277}\cdot 3^{46})^2}=2^{552}$.

If the torsion part of $H^{2k}/(\Sym^kH^2)$ is unknown, then Corollary~\ref{latticeCor} still allows the conclusion, that the square-free parts of the discriminants of $\Sym^kH^2$ and its orthogonal complement are equal.
\end{remark}

\emph{Acknowledgements.} We thank Samuel Boissi\`ere, Cl\'ement Chesseboeuf, K\'evin Tari and Yuan Xu for useful conversations and the University of Poitiers for its hospitality. The author was supported by a DAAD grant.

\bibliographystyle{amsplain}

\begin{thebibliography}{10}

\bibitem{BNS}
Samuel~Boissi\`ere, Marc~Nieper-Wi{\ss}kirchen and Alessandra~Sarti, \emph{Smith theory and 
  Irreducible Holomorphic Symplectic Manifolds}, Journal of Topology 6 (2013), no. 2, 361–-390 

\bibitem{Chihara}
Theodore~S.~Chihara, \emph{An introduction to orthogonal polynomials},
  Mathematics and its Applications 13, Gordon and Breach Science Publishers (1978)

\bibitem{Dai}
Feng~Dai and Yuan~Xu, \emph{Spherical Harmonics}, eprint arXiv:1304.2585 (2013)

\bibitem{Dolgachev}
Igor V.~Dolgachev, \emph{Classical Algebraic Geometry}, 
  Cambridge University Press (2012)

\bibitem{Dunkl}
Charles~F.~Dunkl and Yuan~Xu, \emph{Orthogonal Polynomials of Several Variables},
  Encyclopedia of Mathematics and its Applications, Vol.~81, Cambridge University Press (2001)

\bibitem{Folland}
Gerald~B.~Folland, \emph{How to Integrate a Polynomial over a Sphere}, The American
  Mathematical Monthly, Vol.~108, no.~5, (May,~2001)

\bibitem{Huybrechts}
Mark~Gross, Daniel~Huybrechts, Dominic~Joyce, \emph{Calabi-Yau Manifolds and Related Geometries},
  Universitext, Springer (2003)

\bibitem{Kapfer}
Simon~Kapfer, \emph{Computing Cup-Products in integral cohomology of Hilbert schemes 
  of points on K3 surfaces}, eprint arXiv:1410.8398 (2014)

\bibitem{McGarr}
Se\'an McGarraghy, \emph{Symmetric Powers of Symmetric Bilinear Forms}, 
  Algebra Colloquium 12:1 (2005) 41-57

\bibitem{OGrady}
Kieran~G.~O'Grady, \emph{Compact Hyperk\"ahler manifolds: general theory}, lecture notes (2014)
  \url{www.mimuw.edu.pl/~gael/Document/hk-theory.pdf}

\end{thebibliography}

\end{document}